\documentclass[10pt, reqno]{amsart}

\usepackage{hyperref, cite, xcolor}
\usepackage{caption}
\usepackage{subcaption}
\usepackage{graphicx}
\usepackage{amsmath}
\usepackage{amsthm}
\usepackage{amssymb}
\usepackage{amsfonts}
\usepackage{latexsym}
\usepackage{setspace}
\usepackage{enumitem}
\setlist[itemize]{leftmargin=*}
\setlist[enumerate]{leftmargin=*}

\usepackage[font=footnotesize,labelfont=bf]{caption}

\graphicspath{{./images/}}

\renewcommand{\epsilon}{\varepsilon}

\newtheorem{theorem}{Theorem}
\newtheorem{conjecture}{Conjecture}

\newtheorem{lemma}[theorem]{Lemma}

\theoremstyle{definition}

\numberwithin{equation}{section}

\renewcommand{\phi}{\varphi}

\newcommand{\q}{\mathfrak{q}}

\newcommand{\Z}{\mathbb{Z}}
\renewcommand{\pmod}[1]{\,\,(\operatorname{mod} #1)}

\renewcommand{\geq}{\geqslant}
\renewcommand{\leq}{\leqslant}

\let\oldenumerate=\enumerate
	\def\enumerate{
	\oldenumerate
	\setlength{\itemsep}{5pt}
	}
\let\olditemize=\itemize
	\def\itemize{
	\olditemize
	\setlength{\itemsep}{5pt}
	}

\allowdisplaybreaks

\begin{document}

\title{Primitive root bias for twin primes}

\author[S.R.~Garcia]{Stephan Ramon Garcia}
\address{Department of Mathematics\\Pomona College\\610 N. College Ave., Claremont, CA 91711} 
\email{stephan.garcia@pomona.edu}
\urladdr{http://pages.pomona.edu/~sg064747}
\thanks{SRG supported by NSF grant DMS-1265973,
a David L. Hirsch III and Susan H. Hirsch Research Initiation Grant, 
and the Budapest Semesters in Mathematics (BSM)
Director's Mathematician in Residence (DMiR) program.}

\author[E.~Kahoro]{Elvis Kahoro}

\author[F.~Luca]{Florian Luca}
\address{School of Mathematics\\University of the Witwatersrand\\Private Bag 3, Wits 2050, Johannesburg, South Africa\\
Max Planck Institute for Mathematics, Vivatgasse 7, 53111 Bonn, Germany\\
Department of Mathematics, Faculty of Sciences, University of Ostrava, 30 dubna 22, 701 03
Ostrava 1, Czech Republic}
\email{Florian.Luca@wits.ac.za}
\thanks{F. L. was supported in part by grants CPRR160325161141 and an A-rated researcher award both from the NRF of South Africa and by grant no. 17-02804S of the Czech Granting Agency. }

\begin{abstract}
Numerical evidence suggests that for only about $2\%$ of pairs $p,p+2$ of twin primes, $p+2$ has more primitive roots than does $p$.  If this occurs, we say that $p$ is \emph{exceptional} (there are only two exceptional pairs with $5 \leq p \leq 10{,}000$).  Assuming the Bateman--Horn conjecture, we prove that at least $0.47\%$ of twin prime pairs are exceptional and at least $65.13\%$ are not exceptional.  We also conjecture a precise formula for the proportion of exceptional twin primes.
\end{abstract}



\maketitle

\section{Introduction}

Let $n$ be a positive integer.  An integer coprime to $n$ is a \emph{primitive root} modulo $n$
if it generates the multiplicative group $(\Z/n\Z)^{\times}$ of units modulo $n$.  
A famous result of Gauss states that $n$ possesses primitive roots if and only if
$n$ is $2$, $4$, an odd prime power, or twice an odd prime power.  If a primitive root modulo $n$ exists,
then $n$ has precisely $\phi(\phi(n))$ of them, in which $\phi$ denotes the Euler totient function.
If $p$ is prime, then $\phi(p) = p-1$ and hence $p$ has exactly $\phi(p-1)$ primitive roots.

If $p$ and $p+2$ are prime, then $p$ and $p+2$ are \emph{twin primes}.  
The Twin Prime Conjecture asserts that there are infinitely many twin primes.
While it remains unproved, recent years have seen an explosion of closely-related work
\cite{Polymath, Zhang, Maynard}.
Let $\pi_2(x)$ denote the number of primes $p$ at most $x$ for which $p+2$ is prime.
The first Hardy--Littlewood conjecture asserts that
\begin{equation}\label{eq:HL}
 \pi _{2}(x)\,\sim\, 2C_{2}\int _{2}^{x}{dt \over (\log  t)^{2}},
\end{equation}
in which
\begin{equation}\label{eq:TwinPrimesConstant}
C_2 = \prod_{p \geq 3} \frac{p(p-2)}{(p-1)^2} = 0.660161815\ldots.
\end{equation}
is the \emph{twin primes constant} \cite{Hardy}.  A simpler expression that is asymptotically equivalent to \eqref{eq:HL}
is $2C_{2}x/(\log  x)^2$.

A casual inspection (see Table \ref{Table:PrimRootList})
suggests that if $p$ and $p+2$ are primes and $p \geq 5$, then 
$p$ has at least as many primitive roots as $p+2$; that is, 
$\phi(p-1) \geq \phi(p+1)$.
If this occurs, then $p$ is \emph{unexceptional}.  The preceding inequality
holds for all twin primes $p,p+2$ with
$5 \leq p \leq 10{,}000$, except for 
the pairs $2381, 2383$ and $3851, 3853$.  

\begin{table}\footnotesize
\begin{equation*}
\begin{array}{|c|ccc||c|ccc|}
\hline
p & \phi(p-1) & \phi(p+1) & \delta(p) & 
p & \phi(p-1) & \phi(p+1) & \delta(p) \\
\hline
 5 & 2 & 2 & 0 & 821 & 320 & 272 & 48 \\
 11 & 4 & 4 & 0 & 827 & 348 & 264 & 84 \\
 17 & 8 & 6 & 2 & 857 & 424 & 240 & 184 \\
 29 & 12 & 8 & 4 & 881 & 320 & 252 & 68 \\
 41 & 16 & 12 & 4 & 1019 & 508 & 256 & 252 \\
 59 & 28 & 16 & 12 & 1031 & 408 & 336 & 72 \\
 71 & 24 & 24 & 0 & 1049 & 520 & 240 & 280 \\
 101 & 40 & 32 & 8 & 1061 & 416 & 348 & 68 \\
 107 & 52 & 36 & 16 & 1091 & 432 & 288 & 144 \\
 137 & 64 & 44 & 20 & 1151 & 440 & 384 & 56 \\
 149 & 72 & 40 & 32 & 1229 & 612 & 320 & 292 \\
 179 & 88 & 48 & 40 & 1277 & 560 & 420 & 140 \\
 191 & 72 & 64 & 8 & 1289 & 528 & 336 & 192 \\
 197 & 84 & 60 & 24 & 1301 & 480 & 360 & 120 \\
 227 & 112 & 72 & 40 & 1319 & 658 & 320 & 338 \\
 239 & 96 & 64 & 32 & 1427 & 660 & 384 & 276 \\
 269 & 132 & 72 & 60 & 1451 & 560 & 440 & 120 \\
 281 & 96 & 92 & 4 & 1481 & 576 & 432 & 144 \\
 311 & 120 & 96 & 24 & 1487 & 742 & 480 & 262 \\
 347 & 172 & 112 & 60 & 1607 & 720 & 528 & 192 \\
 419 & 180 & 96 & 84 & 1619 & 808 & 432 & 376 \\
 431 & 168 & 144 & 24 & 1667 & 672 & 552 & 120 \\
 461 & 176 & 120 & 56 & 1697 & 832 & 564 & 268 \\
 521 & 192 & 168 & 24 & 1721 & 672 & 480 & 192 \\
 569 & 280 & 144 & 136 & 1787 & 828 & 592 & 236 \\
 599 & 264 & 160 & 104 & 1871 & 640 & 576 & 64 \\
 617 & 240 & 204 & 36 & 1877 & 792 & 624 & 168 \\
 641 & 256 & 212 & 44 & 1931 & 768 & 528 & 240 \\
 659 & 276 & 160 & 116 & 1949 & 972 & 480 & 492 \\
 809 & 400 & 216 & 184 & 1997 & 996 & 648 & 348 \\
\hline
\end{array}
\end{equation*}
\caption{For twin primes $p,p+2$ with $5 \leq p \leq 2000$, the 
difference $\delta(p) = \phi(p-1) - \phi(p+1)$ is nonnegative.  That is, $p$
has at least as many primitive roots as does $p+2$.}
\label{Table:PrimRootList}
\end{table}

If $p,p+2$ are primes with $p \geq 5$ and $\phi(p-1) < \phi(p+1)$, then $p$ is \emph{exceptional}.
We do not regard $p=3$ as exceptional for technical reasons.
Let $\pi_e(x)$ denote the number of exceptional primes $p \leq x$; that is,
\begin{equation*}
\pi_e(x) \,=\, \# \big\{ p \leq x : \text{$p$ and $p+2$ are prime and $\phi(p-1) < \phi(p+1)$} \big\}.
\end{equation*}
Computational evidence suggests that approximately $2\%$ of twin primes are exceptional;
see Table \ref{Table:First100}.  We make the following conjecture.

\begin{table}\footnotesize
\begin{equation*}
\begin{array}{|c|cccl||c|cccl|}
\hline
p & \delta(p) & \pi_2(p) & \pi_e(p) & \pi_e(p) / \pi_2(p) &
p & \delta(p) & \pi_2(p) & \pi_e(p) & \pi_e(p) / \pi_2(p)  \\
\hline
 2381 & -24 & 71 & 1 & 0.0140845 & 230861 & -2304 & 2427 & 51 & 0.0210136 \\
 3851 & -72 & 100 & 2 & 0.02 & 232961 & -1952 & 2447 & 52 & 0.0212505 \\
 14561 & -240 & 268 & 3 & 0.011194 & 237161 & -784 & 2486 & 53 & 0.0213194 \\
 17291 & -16 & 300 & 4 & 0.0133333 & 241781 & -4232 & 2517 & 54 & 0.0214541 \\
 20021 & -680 & 342 & 5 & 0.0146199 & 246611 & -4440 & 2557 & 55 & 0.0215096 \\
 20231 & -192 & 344 & 6 & 0.0174419 & 251231 & -768 & 2598 & 56 & 0.021555 \\
 26951 & -576 & 430 & 7 & 0.0162791 & 259211 & -1392 & 2657 & 57 & 0.0214528 \\
 34511 & -736 & 532 & 8 & 0.0150376 & 270131 & -3256 & 2755 & 58 & 0.0210526 \\
 41231 & -768 & 602 & 9 & 0.0149502 & 274121 & -5376 & 2788 & 59 & 0.0211621 \\
 47741 & -1152 & 672 & 10 & 0.014881 & 275591 & -1136 & 2800 & 60 & 0.0214286 \\
 50051 & -1728 & 706 & 11 & 0.0155807 & 278741 & -6512 & 2827 & 61 & 0.0215776 \\
 52361 & -2088 & 731 & 12 & 0.0164159 & 282101 & -7632 & 2853 & 62 & 0.0217315 \\
 55931 & -432 & 765 & 13 & 0.0169935 & 282311 & -720 & 2855 & 63 & 0.0220665 \\
 57191 & -912 & 780 & 14 & 0.0179487 & 298691 & -3552 & 2982 & 64 & 0.0214621 \\
 65171 & -552 & 856 & 15 & 0.0175234 & 300581 & -3420 & 3000 & 65 & 0.0216667 \\
 67211 & -312 & 876 & 16 & 0.0182648 & 301841 & -3840 & 3012 & 66 & 0.0219124 \\
 67271 & -96 & 878 & 17 & 0.0193622 & 312551 & -4752 & 3103 & 67 & 0.021592 \\
 70841 & -2492 & 915 & 18 & 0.0196721 & 315701 & -9228 & 3130 & 68 & 0.0217252 \\
 82811 & -720 & 1043 & 19 & 0.0182167 & 316031 & -5376 & 3132 & 69 & 0.0220307 \\
 87011 & -2112 & 1084 & 20 & 0.0184502 & 322631 & -7200 & 3197 & 70 & 0.0218955 \\
 98561 & -2132 & 1207 & 21 & 0.0173985 & 325781 & -6012 & 3230 & 71 & 0.0219814 \\
 101501 & -228 & 1235 & 22 & 0.0178138 & 328511 & -5440 & 3259 & 72 & 0.0220927 \\
 101531 & -240 & 1236 & 23 & 0.0186084 & 330821 & -4284 & 3283 & 73 & 0.0222358 \\
 108461 & -312 & 1302 & 24 & 0.0184332 & 341321 & -2928 & 3354 & 74 & 0.0220632 \\
 117041 & -4452 & 1388 & 25 & 0.0180115 & 345731 & -5088 & 3388 & 75 & 0.022137 \\
 119771 & -912 & 1420 & 26 & 0.0183099 & 348461 & -3348 & 3413 & 76 & 0.0222678 \\
 126491 & -1584 & 1482 & 27 & 0.0182186 & 354971 & -7920 & 3459 & 77 & 0.0222608 \\
 129221 & -2736 & 1508 & 28 & 0.0185676 & 356441 & -4764 & 3473 & 78 & 0.022459 \\
 134681 & -3420 & 1559 & 29 & 0.0186017 & 357281 & -6264 & 3480 & 79 & 0.0227011 \\
 136991 & -1568 & 1586 & 30 & 0.0189155 & 361901 & -10232 & 3520 & 80 & 0.0227273 \\
 142871 & -2688 & 1634 & 31 & 0.0189718 & 362951 & -4080 & 3525 & 81 & 0.0229787 \\
 145601 & -2448 & 1653 & 32 & 0.0193587 & 371141 & -2736 & 3580 & 82 & 0.022905 \\
 150221 & -1688 & 1703 & 33 & 0.0193776 & 399491 & -6048 & 3800 & 83 & 0.0218421 \\
 156941 & -2196 & 1772 & 34 & 0.0191874 & 402221 & -11064 & 3818 & 84 & 0.022001 \\
 165551 & -4768 & 1848 & 35 & 0.0189394 & 404321 & -1584 & 3838 & 85 & 0.022147 \\
 166601 & -1772 & 1855 & 36 & 0.019407 & 406631 & -752 & 3862 & 86 & 0.0222683 \\
 167861 & -3360 & 1869 & 37 & 0.0197967 & 410411 & -15568 & 3887 & 87 & 0.0223823 \\
 173741 & -56 & 1909 & 38 & 0.0199057 & 413141 & -3744 & 3909 & 88 & 0.0225122 \\
 175631 & -3232 & 1924 & 39 & 0.0202703 & 416501 & -4272 & 3934 & 89 & 0.0226233 \\
 188861 & -2472 & 2061 & 40 & 0.0194081 & 418601 & -12812 & 3949 & 90 & 0.0227906 \\
 197891 & -1392 & 2139 & 41 & 0.0191678 & 424271 & -20448 & 3996 & 91 & 0.0227728 \\
 202931 & -3672 & 2179 & 42 & 0.0192749 & 427421 & -1352 & 4026 & 92 & 0.0228515 \\
 203771 & -720 & 2190 & 43 & 0.0196347 & 438131 & -4576 & 4114 & 93 & 0.0226057 \\
 205031 & -1136 & 2204 & 44 & 0.0199637 & 440441 & -20088 & 4120 & 94 & 0.0228155 \\
 205661 & -3288 & 2208 & 45 & 0.0203804 & 448631 & -13536 & 4184 & 95 & 0.0227055 \\
 206081 & -468 & 2211 & 46 & 0.0208051 & 454721 & -1044 & 4232 & 96 & 0.0226843 \\
 219311 & -3936 & 2321 & 47 & 0.0202499 & 464171 & -912 & 4299 & 97 & 0.0225634 \\
 222041 & -1632 & 2347 & 48 & 0.0204516 & 464381 & -2148 & 4302 & 98 & 0.0227801 \\
 225611 & -5088 & 2381 & 49 & 0.0205796 & 465011 & -9840 & 4309 & 99 & 0.0229752 \\
 225941 & -432 & 2385 & 50 & 0.0209644 & 470471 & -24336 & 4341 & 100 & 0.0230362 \\
 \hline
\end{array}
\end{equation*}
\caption{The first $100$ exceptional $p$.  Here $\delta(p) = \phi(p-1) - \phi(p+1)$. }
\label{Table:First100}
\end{table}

\begin{conjecture}\label{Conjecture:Main}
A positive proportion of the twin primes are exceptional.  That is,
$\lim_{x \to \infty} \pi_e(x) / \pi_2(x)$ exists and is positive.
\end{conjecture}

\begin{figure}
\includegraphics[width=\textwidth]{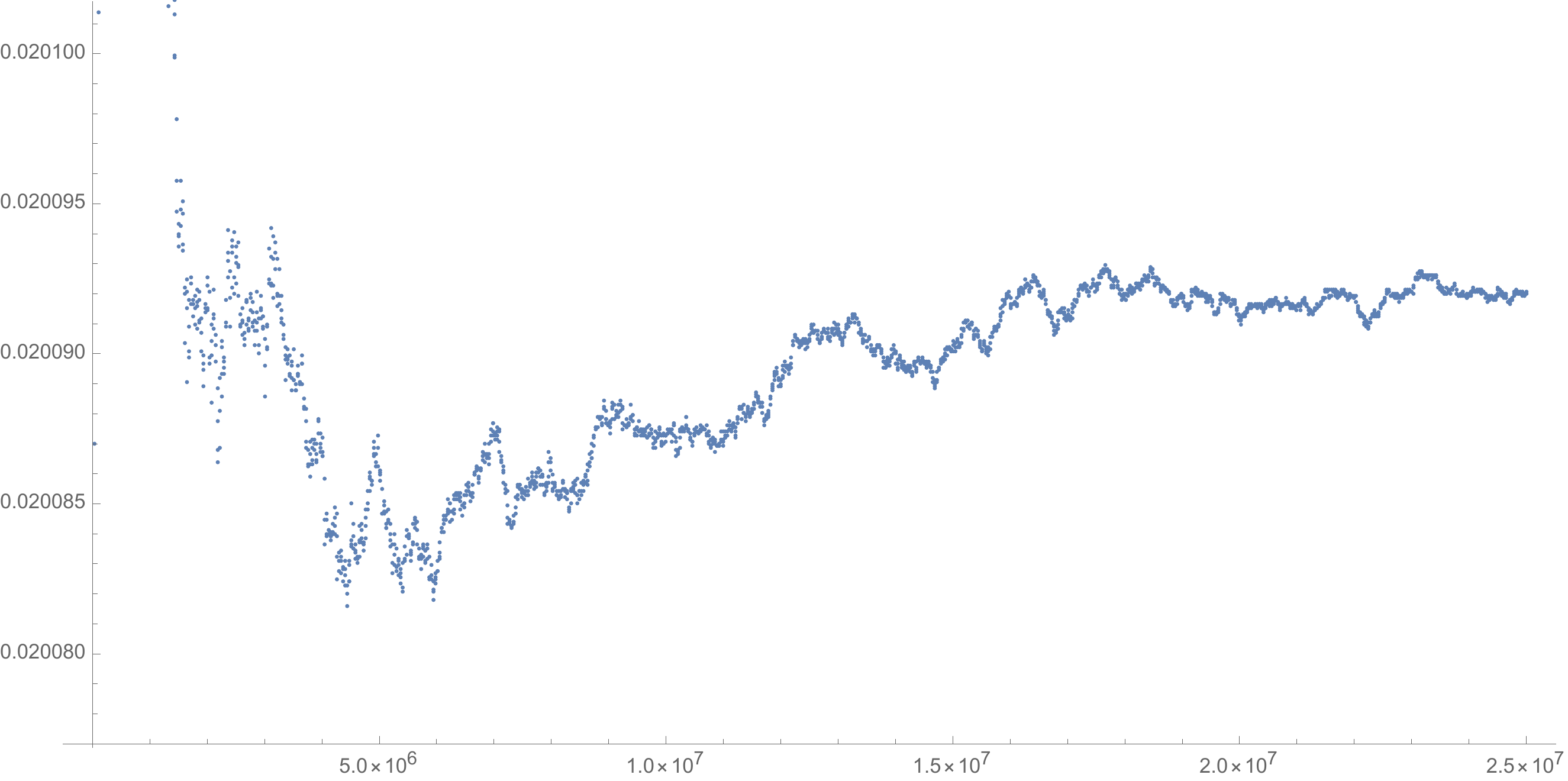}
\caption{Numerical evidence suggests that $\lim_{x\to\infty}\pi_e(x) / \pi_2(x)$ exists and is slightly larger than $2\%$.  The horizontal axis denotes the number of exceptional twin prime pairs.  The vertical axis represents the ratio $\pi_e / \pi_2$.}
\label{Figure:Ratio}
\end{figure}

We are able to prove Conjecture \ref{Conjecture:Main}, if we assume
the Bateman--Horn conjecture (stated below).  Our main theorem is the following.

\begin{theorem}\label{Theorem:Main}
Assume that the Bateman--Horn conjecture holds. 
\begin{enumerate}
\item The set of twin prime pairs $p, p+2$ for which $\phi(p-1) < \phi(p+1)$ has lower density (as a subset of twin primes) at least  $0.47\%$.
\item The set of twin prime pairs $p, p+2$ for which $\phi(p-1) \geq \phi(p+1)$ has lower density (as a subset of twin primes) at least  $65.13\%$.
\end{enumerate}
\end{theorem}

Computations suggest that the value of the limit in Conjecture \ref{Conjecture:Main}
is approximately $2\%$;
see Figures \ref{Figure:Ratio} and \ref{Figure:BlueRed}.
A value for the limiting ratio is proposed in Section \ref{Section:Conjecture}.

It is also worth pointing out that 
this bias is specific to the twin primes since
the set of primes $p$ for which $\phi(p-1) - \phi(p+1)$ is positive (respectively, negative) has density $50\%$ as a subset of the primes \cite{GaLu}.  That is, if we remove the assumption that $p+2$ is
also prime, then the bias completely disappears.  Although only tangentially related to the present discussion, it is worth
mentioning the exciting preprint \cite{Sound} which concerns a peculiar and unexpected bias in the primes.

\section{The Bateman--Horn conjecture}

The proof of Theorem \ref{Theorem:Main} is deferred until Section \ref{Section:Proof}.
We first require a few words about the Bateman--Horn conjecture.
Let $f_1,f_2,\ldots,f_m$ be a collection of distinct irreducible polynomials 
with positive leading coefficients.  An integer $n$ is \emph{prime generating} for 
this collection if each $f_1(n), f_2(n),\ldots, f_m(n)$ is prime. 
Let $P(x)$ denote
the number of prime-generating integers at most $x$ and
suppose that $f = f_1f_2\cdots f_m$ does not vanish identically modulo any prime.
The \emph{Bateman--Horn conjecture} is
\begin{equation*}
P(x)\, \sim\,  \frac{C}{D} \int _{2}^{x}\frac{dt}{(\log t)^m},
\end{equation*}
in which 
\begin{equation*}
D = \prod_{i=1}^m \deg f_i
\quad \text{and} \quad
C=\prod_{p} \frac{1-N_f(p)/p}{(1-1/p)^{m}},
\end{equation*}
where $N_f(p)$ is the number of solutions to $f(n) \equiv 0 \pmod{p}$ \cite{Bateman}.

If $f_1(t) = t$ and $f_2(t) = t+2$, then $f(t) = t(t+2)$, $N_f(2) = 1$, and
$N_f(p) = 2$ for $p \geq 3$.  In this case, Bateman--Horn predicts
\eqref{eq:HL}, the first Hardy--Littlewood conjecture, which in turn 
implies the Twin Prime Conjecture.

Although weaker than the Bateman--Horn conjecture, the Brun sieve \cite[Thm.~3, Sect.~I.4.2]{Tenenbaum}
has the undeniable advantage of being proven.  It says that there
exists a constant $B$ that depends only on $m$ and $D$ such that
\begin{equation*}
P(x)\leq  \frac{BC}{D}\int_2^x \frac{dt}{(\log t)^m}
=(1+o(1))\frac{BC}{D} \frac{x}{(\log x)^m} 
\end{equation*}
for sufficiently large $x$.  
In particular, 
\begin{equation*}
\pi_2(x)\leq \frac{K x}{(\log x)^2}
\end{equation*}
for some constant $K$ and sufficiently large $x$. 
The best known $K$ in the estimate above is $K=4.5$ \cite{Jie}.

	\begin{figure}
		\centering
		\begin{subfigure}[b]{0.475\textwidth}
	                \centering
	                \includegraphics[width=\textwidth]{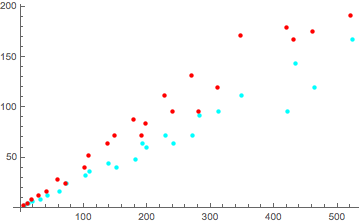}
	                \caption{{\scriptsize First 500 twin primes}}
	                \label{fig:TP03}
	        \end{subfigure}
	        \quad
	        		\begin{subfigure}[b]{0.475\textwidth}
	                \centering
	                \includegraphics[width=\textwidth]{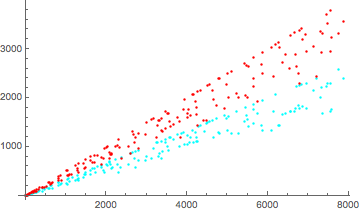}
	                \caption{{\scriptsize First 8{,}000 twin primes}}
	                \label{fig:TP06}
	        \end{subfigure}
	        \\
		\begin{subfigure}[b]{0.475\textwidth}
	                \centering
	                \includegraphics[width=\textwidth]{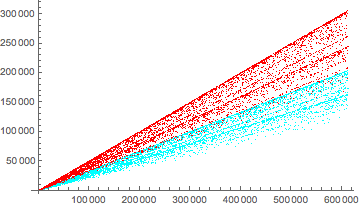}
	                \caption{{\scriptsize First 600{,}000 twin primes}}
	                \label{fig:TP09}
	        \end{subfigure}
	        \quad
	        		\begin{subfigure}[b]{0.475\textwidth}
	                \centering
	                \includegraphics[width=\textwidth]{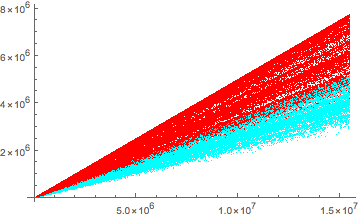}
	                \caption{{\scriptsize First 1.5 million twin primes}}
	                \label{fig:TP12}
	        \end{subfigure}
	        \caption{Plots in the $xy$-plane of ordered pairs
	        		{\color{red} $(p,\phi(p-1))$} (in red) and {\color{cyan}$(p+2,\phi(p+1))$}
			(in cyan) for twin primes $p,p+2$.  There are no exceptional pairs visible in 
			Figure \ref{fig:TP03}; that is, $\phi(p-1) \geq \phi(p+1)$ in each case.  			
			The exceptional pairs
			$2381, 2383$ and $3851, 3853$ are visible in Figure \ref{fig:TP06}.
			A smattering of exceptional pairs emerge as more twin primes are considered.}
		\label{Figure:BlueRed}
	\end{figure}

\section{An heuristic argument}

We give an heuristic argument which suggests that 
$\phi(p-1) \geq \phi(p+1)$ for an overwhelming proportion of twin primes $p, p+2$. 
It also identifies specific conditions under which $\phi(p-1) < \phi(p+1)$ might occur.  
This informal reasoning can be made rigorous under the assumption of the Bateman--Horn conjecture
(see Section \ref{Section:Proof}).

Observe that each pair of twin primes, aside from $3,5$, is of the form $6n-1,6n+1$.  Thus, if $p,p+2$
are twin primes with $p \geq 3$, then $2|(p-1)$ and $6|(p+1)$.  We use this in the following lemma
to obtain an equivalent characterization of (un)exceptionality.

\begin{lemma}\label{Lemma:Trouble}
If $p$ and $p+2$ are prime and $p \geq 5$, then
\begin{equation}\label{eq:Trouble}
\phi(p-1) \,\geq \,\phi(p+1)
\quad\iff\quad
\frac{\phi(p-1)}{p-1}\, \geq \,\frac{\phi(p+1)}{p+1}.
\end{equation}
\end{lemma}

\begin{proof}
The forward implication is straightforward arithmetic, so we focus on the reverse.
If the inequality on the right-hand side of \eqref{eq:Trouble} holds, then
\begin{align*}
0 
&\leq p \big(\phi(p-1) - \phi(p+1)\big) + \phi(p-1) + \phi(p+1) \\
&\leq p \big(\phi(p-1) - \phi(p+1)\big) + \tfrac{1}{2}(p-1) + \tfrac{1}{3}(p+1) \\
&< p \big(\phi(p-1) - \phi(p+1)\big) + \tfrac{5}{6}p
\end{align*}
since $2|(p-1)$ and $6|(p+1)$.  For the preceding to hold,
the integer $\phi(p-1) - \phi(p+1)$ must be nonnegative.
\end{proof}

In light of \eqref{eq:Trouble} and the formula (in which $q$ is prime)
\begin{equation*}
\frac{\phi(n)}{n}  = \prod_{q|n} \left(1- \frac{1}{q} \right),
\end{equation*}
it follows that $p$ is exceptional if and only if $p+2$ is prime and
\begin{equation}\label{eq:Condition}
 \frac{1}{2}\prod_{\substack{q|(p-1)\\q\geq 5}} \left(1 - \frac{1}{q} \right)
\quad<\quad
\frac{1}{3}\prod_{\substack{q|(p+1)\\q\geq 5}} \left(1 - \frac{1}{q} \right)
\end{equation}
because $2|(p-1)$, $3\nmid(p-1)$ and $6|(p+1)$.
The condition \eqref{eq:Condition} can occur if $p-1$ is divisible by only small primes.
For example, if $5,7,11 | (p-1)$, then $5,7,11 \nmid (p+1)$ and the quantities in \eqref{eq:Condition} become
\begin{equation*}
\frac{24}{77}\prod_{\substack{q|(p-1)\\q\geq 13}} \left(1 - \frac{1}{q} \right)
\quad \text{and} \quad
\frac{1}{3}\prod_{\substack{q|(p+1)\\q\geq 13}} \left(1 - \frac{1}{q} \right).
\end{equation*}
Since
\begin{equation*}
\frac{24}{77} \approx 0.3117 < \frac{1}{3}
\qquad \text{and} \qquad 
2 \cdot 5 \cdot 7 \cdot 11 = 770,
\end{equation*}
one expects \eqref{eq:Condition} to hold occasionally if $p = 770n + 1$.
Dirichlet's theorem on primes in arithmetic progressions ensures that $p+2 = 770n+3$ is prime
$1/\phi(770) = 1/240 = 0.4167\%$ of the time.  Thus, we expect a small proportion of twin prime pairs to
satisfy \eqref{eq:Condition}.
For example,
among the first $100$ exceptional pairs (see Table \ref{Table:First100}), the following values of $p$ have the form $770n+1$:
\begin{quote}
3851, 20021, 26951, 47741, 50051, 52361, 70841, 87011, 98561, 117041, 165551, 167861, 197891, 225611, 237161, 241781, 274121, 278741, 301841, 315701, 322631, 345731, 354971, 357281, 361901, 371141, 410411, 424271, 438131, 440441, 470471.
\end{quote}
This accounts for $31\%$ of the first $100$ exceptional pairs.  We now make this heuristic argument rigorous,
under the assumption that the Bateman--Horn conjecture holds.

\section{Proof of Theorem \ref{Theorem:Main}}\label{Section:Proof}

Assume that the Bateman--Horn conjecture holds. We first prove statement
(a) of Theorem \ref{Theorem:Main}.
In what follows, $p,q,r$ denote prime numbers.
\medskip

\noindent\textbf{Proof of (a).}
Consider twin primes $p,p+2$ such that $5,7,11|(p-1)$. 
Let $\pi_{2}'(x)$ be the number of such $p\leq x$.
\medskip

\noindent\textbf{Step 1.}
Since $5 \cdot 7 \cdot 11 = 385$,
the desired primes are precisely those of the form
\begin{equation*}
\text{$n = 385k+1\leq x$ \quad such that \quad $n+2 = 385 k + 3$ is prime}.
\end{equation*}
In the Bateman--Horn conjecture, let 
\begin{equation*}
f_1(t)=385 t+1, \qquad f_2(t)=385 t+3, \quad \text{and} \quad f = f_1 f_2.
\end{equation*}
Then
\begin{equation}\label{eq:Nf1}
N_f(p) = 
\begin{cases}
1 & \text{if $p=2$},\\ 
2 & \text{if $p=3$},\\ 
0 & \text{if $p=5,7,11$},\\
2 & \text{if $p \geq 13$}.
\end{cases}
\end{equation}
Since $p\leq x$, we must have $k\leq (x-1)/385$.  
For sufficiently large $x$, the Bateman--Horn conjecture predicts that the number of such $k$ is
\begin{align}
\pi_{2}'(x) 
& =  (1+o(1)) \frac{(x-1)/385}{(\log ((x-1)/385))^2} \prod_{p\geq 2} \left(\frac{1-N_f(p)/p}{(1-1/p)^2}\right) \nonumber\\
& =  (1+o(1)) \left(\frac{2x}{385(\log x)^2}\right) \prod_{p\geq 3} \left(\frac{1-N_f(p)/p}{(1-1/p)^2}\right)\nonumber\\
& =  (1+o(1)) \left(\frac{2x}{385(\log x)^2}\right) \prod_{p=5,7,11} \left( \frac{1}{(1- 1/p)^2} \right) \prod_{\substack{p\geq 13\\\text{or $p =3$}}} \left( \frac{1 - 2/p}{ (1 -1/p)^2} \right)\nonumber \\
& =  (1+o(1)) \left(\frac{2x}{385(\log x)^2}\right)  \prod_{p\geq 3} \left( \frac{1 - 2/p}{ (1 -1/p)^2} \right) \prod_{p=5,7,11} \left( 1 - 2/p\right)^{-1}\nonumber\\
& =  (1+o(1)) \left(\frac{2x}{385(\log x)^2}\right)   \prod_{p\geq 3} \left( \frac{p(p-2)}{ (p-1)^2} \right) \frac{5 \cdot 7\cdot 11}{(5-2)(7-2)(11-2)}\nonumber \\
& =  (1+o(1)) \frac{ 2C_2 x}{135(\log x)^2} \nonumber \\
& =  (1+o(1)) \frac{\pi_2(x)}{135} \nonumber\\
&>0.00740740\, \pi_2(x). \label{eq:p2p}
\end{align}

\noindent\textbf{Step 2.}
Fix a prime $r \geq 13$.
Let $\pi_{2,r}'(x)$ be the number of primes $p\leq x$ such that
$p,p+2$ are prime, $5,7,11|(p-1)$, and $r|(p+1)$. 
The desired primes are precisely those of the form
\begin{equation*}
\text{$n = 385 k+1\leq x$  \quad such that \quad $n+2 = 385 k + 3$ is prime
and $r|(385 k + 2)$}.
\end{equation*}
In particular, $k$ must be of the form
\begin{equation*}
k = k_0 + r \ell, 
\end{equation*}
in which $k_0$ is the smallest positive integer with $k_0\equiv -2(385)^{-1}  \pmod r$.  
Let $b_{r}=385 k_0+1$.  Then
\begin{equation}\label{eq:APr}
n=385 r \ell+ b_r
\qquad \text{and}\qquad
n+2=385 r\ell+(b_r + 2),
\end{equation}
are both prime, $n\leq x$, and
\begin{equation*}
\ell\leq \frac{x-b_{r}}{385 r}  .
\end{equation*}
In the Bateman--Horn conjecture, let 
\begin{equation*}
f_1(t)=385 rt+b_{r},  \qquad f_2(t)=385 r t+(b_{r}+2),
\quad\text{and}\quad f = f_1f_2.
\end{equation*}
Then $N_f(p)$ is as in \eqref{eq:Nf1} except for $p=r$, in which case $N_f(r) = 0$.
Indeed, 
\begin{equation*}
f_1(t) \equiv b_r \equiv 385k_0 + 1 \equiv -1 \pmod{r}
\quad\text{and}\quad
f_2(t) \equiv b_r + 2 \equiv 1 \pmod{r}
\end{equation*}
for all $t$.
As $x \to \infty$, the Bateman--Horn conjecture 
predicts that the number of such $\ell$ is
\begin{align}
\pi_{2,r}'(x) 
& =  (1+o(1)) \frac{(x-b_r)/(385 r)}{( \log( (x - b_r)/(385 r)))^2 } \prod_{p \geq 2} \left( \frac{1 - N_f(p)/p}{(1 - 1/p)^2} \right) \nonumber \\
&=  (1+o(1)) \frac{x}{ 385r (\log x )^2 } \prod_{p \geq 2} \left( \frac{1 - N_f(p)/p}{(1 - 1/p)^2} \right) \nonumber \\
&=  (1+o(1)) \frac{2x}{ 385r (\log x )^2 } \prod_{p \geq 3} \left( \frac{1 - N_f(p)/p}{(1 - 1/p)^2} \right) \nonumber \\
& =  (1+o(1)) \frac{2x}{385 r (\log x)^2} \prod_{p=5,7,11,r} \left( \frac{1}{(1- 1/p)^2} \right) \prod_{p=5,7,11,r} \left( \frac{1 - 2/p}{ (1 -1/p)^2} \right) \nonumber\\
& =  (1+o(1)) \left(\frac{2x}{385r(\log x)^2}\right)   \prod_{p\geq 3} \left( \frac{p(p-2)}{ (p-1)^2} \right) \frac{5 \cdot 7\cdot 11 \cdot r}{(5-2)(7-2)(11-2)(r-2)} \nonumber \\
& =  (1+o(1)) \frac{2C_2x}{135(r-2)(\log x)^2} \nonumber \\
& =  (1+o(1)) \frac{\pi_2(x)}{135(r-2)}.\label{eq:Pi2xphi2ayr}
\end{align}

\noindent\textbf{Step 3.}
Suppose that $p$ is counted by $\pi_{2}'(x)$; that is,
suppose that $p,p+2$ are prime and that $5,7,11 | (p-1)$.  Then
$6|(p+1)$, $5,7,11 \nmid (p+1)$, and
\begin{equation*}
\frac{\phi(p-1)}{p-1}\leq \prod_{q=2,5,7,11} \left(1-\frac{1}{q}\right)=\frac{24}{77}.
\end{equation*}
If the pair $p$ is unexceptional, then Lemma \ref{Lemma:Trouble} ensures that
\begin{equation*}
 \frac{1}{3}\prod_{\substack{r\mid (p+1)\\ r\geq 13}} \left(1-\frac{1}{r}\right) = \frac{\phi(p+1)}{p+1} \leq \frac{\phi(p-1)}{p-1}\leq  \frac{24}{77}.
\end{equation*}
Consequently,
\begin{equation*}
\prod_{\substack{r| (p+1)\\ r\geq 13}} \left(1+\frac{1}{r-1}\right) \geq  \frac{77}{72},
\end{equation*}
in which $r$ is prime. Let
\begin{equation*}
F(p)
=\sum_{\substack{r| (p+1)\\ r\geq 13}} \log\left(1+\frac{1}{r-1}\right).
\end{equation*}

\noindent\textbf{Step 4.}
We want to count the twin primes pairs $p,p+2$ with $p\leq x$, $F(p) \geq \log (77/72)$, and
$5,7,11|(p-1)$. 
To do this, we sum up $F(p)$ over all twin primes $p$ counted by $\pi_{2}'(x)$ and change the order of summation to get
\begin{align}
A(x)
&=\sum_{\substack{\text{$p$ counted by}\\ \pi_{2}'(x)}} F(p) \nonumber\\
&=\sum_{r \geq 13} \pi_{2,r}'(x)  \log\left(1+\frac{1}{r-1}\right) \nonumber\\
& \leq  \sum_{13\leq r\leq z} \pi_{2,r}'(x) \log \left(1+\frac{1}{r-1}\right)\nonumber\\
&\qquad\qquad  +  \sum_{z<r\leq (\log x)^3} \pi_{2,r}'(x)\log\left(1+\frac{1}{r-1}\right)\nonumber\\
 & \qquad\qquad\qquad +  \sum_{(\log x)^3<r\leq x} \pi_{2,r'}(x)\log\left(1+\frac{1}{r-1}\right) \nonumber\\
 & = A_1(x)+A_2(x)+A_3(x), \label{eq:AAA}
\end{align}
in which $z$ is to be determined later.  We bound the three summands separately.
\begin{enumerate} 
\item If $13\leq r\leq z$, then \eqref{eq:Pi2xphi2ayr} asserts that
\begin{equation*}
 \pi_{2,r}'(x) = (1+o(1))\frac{\pi_2(x)}{135(r-2)}
\end{equation*}
uniformly for $r \in [13,z]$ as $x \to \infty$. 
For sufficiently large $x$ we have\footnote{Since $\log (1+t) \leq t$ for $t > 0$,
the terms of the series are $O(1/r^2)$ and hence it converges rapidly enough for reliable numerical evaluation.  \texttt{Mathematica} provides
the value $0.0241503330316$.} 
\begin{align*}
A_1(x) 
& \leq  (1+o(1)) \frac{\pi_2(x)}{135} \left(\sum_{13\leq r\leq z} \frac{1}{(r-2)} \log\left(1+\frac{1}{r-1}\right)\right)\\
& \leq  (1+o(1)) \frac{\pi_2(x)}{135} \left(\sum_{r \geq 13} \frac{1}{(r-2)} \log\left(1+\frac{1}{r-1}\right)\right)\\
& \leq  (1+o(1))  \frac{0.0241504}{135} \pi_2(x) \\
& <  0.000178892 \,\pi_2(x).
\end{align*}

\item If $z<r\leq (\log x)^3$, we use the Brun sieve and manipulations similar to those used to obtain
\eqref{eq:Pi2xphi2ayr} to find an absolute constant $K$ such that
\begin{equation*}
\pi_{2,r}'(x) \leq  \frac{K(x/(135r))}{(\log(x/(135r)))^2} 
\end{equation*}
for sufficiently large $x$.  Since $r\leq (\log x)^3$, 
\begin{equation*}
\log(x/(135 r))\geq \log (x^{1/2})\geq (\log x)/2
\end{equation*}
holds if $x \geq 10^{14}$.  Then \eqref{eq:HL} ensures that
\begin{equation*}
\pi_{2,r}'(x) \leq \frac{4Kx}{135r (\log x)^2} \leq \frac{5K \pi_{2}(x)}{135(r-2)}
\end{equation*}
for sufficiently large $x$.
Now we fix $z$ such that $5K/(135(z-2))< 10^{-9}$.
Since $\log(1+t)<t$ for $t>0$, for sufficiently large $x$ we obtain
\begin{align*}
A_2(x) 
&= \sum_{z<r\leq (\log x)^3} \pi_{2,r}'(x)\log\left(1+\frac{1}{r-1}\right) \\
& \leq  \frac{5K \pi_2(x)}{135} \sum_{r>z} \frac{1}{r-2} \log\left(1+\frac{1}{r-1}\right) \\
& <   \frac{5K\pi_2(x)}{135} \sum_{r>z} \frac{1}{(r-2)(r-1)}\\
&  =  \frac{ 5K \pi_2(x)}{135}  \sum_{r>z} \left(\frac{1}{r-2}-\frac{1}{r-1}\right)\\
& \leq  \frac{5K\pi_2(x)}{135(z-2)} \\
&< 10^{-9} \,\pi_2(x).
\end{align*}

\item Suppose that $(\log x)^3<r\leq x$. 
By \eqref{eq:APr}, the primes counted by $\pi_{2,r}'(x)$ lie in 
an arithmetic progression modulo $385r$.  Thus, their number is at most 
\begin{equation*}
\pi_{2,r}(x)\leq \left\lfloor \frac{x}{385 r}\right\rfloor+1\leq \frac{x}{385 r}+1.
\end{equation*}
Since $\log(1+t)<t$, for sufficiently large $x$ we obtain
\begin{align*}
A_3(x) 
&= \sum_{(\log x)^3<r\leq x} \pi_{2,r'}(x)\log\left(1+\frac{1}{r-1}\right) \\
& \leq   \sum_{(\log x)^3<r\leq x} \frac{1}{(r-1)} \left(\frac{x}{385 r}+1\right)\\
& \leq  \frac{x}{385} \sum_{r>(\log x)^3} \frac{1}{r(r-1)}+\sum_{(\log x)^3<r\leq x} \frac{1}{r-1}\\
& \leq  \frac{x}{385} \sum_{r > (\log x)^3} \left(\frac{1}{r-1}-\frac{1}{r}\right)+\int_{(\log x)^3-2}^{x} \frac{dt}{t}\\
& \leq  \frac{x}{385((\log x)^3-1)}+\left(\log t\Big|_{t=(\log x)^3-2}^{t=x}\right)\\
& \leq  \frac{2x}{385 (\log x)^3}+\log x \\
& =  \left( \frac{1}{385 C_2 \log x} +   \frac{(\log x)^3}{2C_2x} \right)\frac{2C_2 x}{(\log x)^2}\\
&= (1+o(1))\left(  \frac{1}{385 C_2 \log x} + \frac{ (\log x)^3}{2C_2 x} \right) \pi_2(x) \\
&<10^{-9}\, \pi_2(x).
\end{align*}
\end{enumerate} 

\noindent\textbf{Step 5.}
Returning to \eqref{eq:AAA} and using the preceding three estimates, we have
\begin{align*}
A(x) 
& =  A_1(x)+A_2(x)+A_3(x)\\
& <   0.000178892\, \pi_2(x)+ 10^{-9}\, \pi_2(x)+ 10^{-9}\, \pi_2(x)\\
&<   0.000179\, \pi_2(x).
\end{align*}
for sufficiently large $x$.
\medskip

\noindent\textbf{Step 6.}
Let $\mathcal{U}(x)$ be the set of primes $p$ counted by $\pi_{2}'(x)$ that are unexceptional;
that is, $\phi(p-1)/(p-1)\geq \phi(p+1)/(p+1)$ by Lemma \ref{Lemma:Trouble}.  As we have seen,
if $p\in \mathcal{U}(x)$, then $F(p) \geq \log(77/72)$. Thus,
\begin{equation*}
0 \leq \#\mathcal{U}(x) \log(77/72)\leq \sum_{p\in \mathcal{U}(x)} F(p)\leq A(x)\leq 0.000179\, \pi_2(x),
\end{equation*}
from which we deduce that
\begin{equation*}
\# \mathcal{U}(x)\leq \left(\frac{0.000179}{\log(77/72)}\right) \pi_2(x) <0.002667\, \pi_2(x).
\end{equation*}
The primes $p$ counted by $\pi_{2}'(x)$ which are not in $\mathcal{U}(x)$ are exceptional; that is
$\phi(p-1)/(p-1)<\phi(p+1)/(p+1)$.  By \eqref{eq:p2p} and the preceding calculation,
for large $x$ there are at least
\begin{align*}
\pi_2'(x)-\#\mathcal{U}(x) 
& >  \left( 0.00740740 -0.002667 \right) \pi_2(x) \\
&> 0.0047\, \pi_2(x)
\end{align*}
such primes.  This completes the proof of statement (a) from Theorem \ref{Theorem:Main}.

\medskip
\noindent\textbf{Proof of (b).}
This is similar to the preceding, although it is much simpler.
As before, $p,q,r$ denote primes.
If $p,p+2$ are prime and $p$ is exceptional, then
\begin{equation*}
\frac{1}{2} \prod_{ \substack{ r | (p-1) \\ r \geq 5}} \left( 1 - \frac{1}{r} \right) = \frac{ \phi(p-1)}{p-1} 
\leq \frac{ \phi(p+1) }{p+1} \leq \frac{1}{3}
\end{equation*}
since $3\nmid(p-1)$ and $6|(p+1)$.  If we let
\begin{equation*}
G(p)=\sum_{\substack{r|(p-1)\\r\geq 5}} \log\left(1+\frac{1}{r-1}\right),
\end{equation*}
then $G(p)\geq \log(3/2)$ holds for all exceptional primes $p$.
Let $\pi_e(x)$ denote the number of exceptional primes $p \leq x$.
Then
\begin{align*}
\pi_e(x)\log(3/2) 
& \leq  \sum_{\substack{\text{$p$ counted}\\\text{by $\pi_2(x)$}}} G(p) \\
& = \sum_{\substack{\text{$p$ counted}\\\text{by $\pi_2(x)$}}} \sum_{\substack{r\geq 5\\ r|(p-1)}} \log\left(1+\frac{1}{r-1}\right)\\
& \leq  \sum_{5\leq r\leq x} \log\left(1+\frac{1}{r-1}\right) \sum_{\substack{p~{\text{\rm counted~by}}~\pi_2(x)\\ p\equiv 1\pmod r}} 1\\
& \leq  (1+o(1)) \pi_2(x)\sum_{r\geq 5} \frac{1}{(r-2)}\log\left(1+\frac{1}{r-1}\right)\\
& <  0.14137\,\pi_2(x),
\end{align*}
which shows that there are at least
\begin{equation*}
\pi_2(x)- \pi_e(x)\geq \pi_2(x)\left(1-\frac{0.14137}{\log(3/2)} \right)>0.6513\,\pi_2(x)
\end{equation*}
unexceptional primes at most $x$. \qed

\section{Conjectured density}\label{Section:Conjecture}

Below we conjecture a value for the density of the exceptional primes relative to the twin primes.
In what follows, we let
$P(n)$ denote the largest prime factor of $n$ and let
$p(n)$ denote the smallest.  We let $\mu$ denote the
M\"obius function and remind the reader that $\mu^2(n) = 1$ if and only if $n = 1$ or $n$ is the product
of distinct primes.

\begin{conjecture}\label{Conjecture:Formula}
The density of the exceptional twin primes is 
\begin{equation}\label{eq:Conjecture}
\lim_{x\to\infty} \frac{\pi_e(x)}{\pi_2(x)}=
\lim_{\epsilon\to 0} \prod_{5\leq q\leq \frac{1}{\epsilon}} \left(\frac{q-4}{q-2}\right) 
\Bigg(\!\!\!\!\!\!\!\!\!\!\!\!\sum_{\substack{a,b\\ \mu^2(ab)=1\\ 5\leq p(ab)\leq P(ab)\leq \frac{1}{\epsilon} \\ \frac{\phi(a)}{2a}\leq \frac{\phi(b)}{3b}}} \!\!\!\!\!\!\!\! \prod_{p\mid ab} \left(
\frac{1}{p-4}\right)\Bigg).
\end{equation}
\end{conjecture}

A few remarks about the imposing expression \eqref{eq:Conjecture} are in order.
First of all, 
for each fixed $\epsilon>0$, the sum involves only finitely many pairs $a,b$.
Indeed, the condition $\mu^2(ab) =1$ ensures that $ab$ is a product of distinct prime factors.
The restriction $5 \leq p(ab) \leq P(ab) \leq \frac{1}{\epsilon}$ implies that only finitely many prime
factors are available to form $a$ and $b$.  In principle, the right-hand side of \eqref{eq:Conjecture}
can be evaluated to arbitrary accuracy by taking $\epsilon$ sufficiently small.
Unfortunately, the number of terms involved in the sum grows rapidly as $\epsilon$ shrinks
and we are unable to obtain a reliable numerical estimate from \eqref{eq:Conjecture}.

As a brief ``sanity check,'' we also remark that the limit in \eqref{eq:Conjecture},
if it exists, is at most $1$.  Without the condition 
\begin{equation*}
\frac{\phi(a)}{2a}\leq \frac{\phi(b)}{3b} ,
\end{equation*}
the inner sum in \eqref{eq:Conjecture} is
\begin{align*}
\sum_{\substack{a,b\\ \mu^2(ab)=1\\ 5\leq p(ab)\leq P(ab)\leq \frac{1}{\epsilon}}}
\!\!\!\!\!\!\!\!\prod_{p\mid ab} \bigg(\frac{1}{p-4}  \bigg) \qquad
& =  \sum_{\substack{n\\ \mu^2(n)=1\\ 5\leq p(n)\leq P(n)\leq \frac{1}{\epsilon}}} 
\!\!\!\!\!\!\!\!2^{\omega(n)}\prod_{p\mid n} \bigg(\frac{1}{p-4}\bigg)\\ 
& =  \prod_{5\leq p\leq \frac{1}{\epsilon}} \left(1+\frac{2}{p-4}\right)\\
& =  \prod_{5\leq p\leq \frac{1}{\epsilon}} \left(\frac{p-2}{p-4}\right),
\end{align*}
which precisely offsets the first product in \eqref{eq:Conjecture}.  

To proceed, we need to generalize the functions $F$ and $G$ that appeared
in the proof of Theorem \ref{Theorem:Main}.  Let $\epsilon>0$ and define
\begin{equation*}
F_{\epsilon}(p)
= \sum_{\substack{r\mid (p+1)\\ r\geq \frac{1}{\epsilon}}} \log\left(1+\frac{1}{r-1}\right)
\qquad\text{and}\qquad 
G_{\epsilon}(p)
=\sum_{\substack{r\mid (p-1)\\ r\geq \frac{1}{\epsilon}}} \log\left(1+\frac{1}{r-1}\right).
\end{equation*}
Particular instances of these functions have appeared in the proof of Theorem \ref{Theorem:Main} with $\epsilon=1/5$ for $F_{\epsilon}$ (called $F$) and $\epsilon=1/13$ for $G_{\epsilon}$ (called $G$), respectively.

\begin{lemma}
For $\epsilon>0$, the number of twin primes $p\leq x$ such that 
$F_{\epsilon}(p)>\epsilon$
is $O((\log (\frac{1}{\epsilon}))^{-1}\pi_2(x))$. The same conclusion holds with $F_{\epsilon}$ replaced by $G_{\epsilon}$.
\end{lemma}

\begin{proof}
The argument is essentially already in the proof of Theorem \ref{Theorem:Main}. 
We do it only for $F_{\epsilon}(p)$ since 
the argument for $G_{\epsilon}(p)$ is similar. We sum $F_{\epsilon}(p)$ for $p\leq x$ with $p,p+2$ prime and use the fact that $\log(1+t)\leq t$ to obtain 
\begin{align*}
\sum_{\substack{p\leq x\\ \text{$p,p\!+\!2$ prime}}} F_{\epsilon} (p) 
\quad
& \leq  \sum_{\substack{p\leq x\\ \text{$p,p\!+\!2$ prime}}} 
\sum_{\substack{q\mid (p-1)\\ q> \frac{1}{\epsilon}}} \frac{1}{q-1}\\
& =  \sum_{q>\frac{1}{\epsilon}} \frac{1}{q-1}
\sum_{\substack{\text{$p,p\!+\!2$ prime} \\ p\equiv 1\pmod q}} 1\\
& =  \sum_{q>\frac{1}{\epsilon}} \frac{\pi_2(x,q, 1)}{q-1},
\end{align*}
in which $\pi_2(x;q,1)$ denotes the number of primes $p\leq x$ with $p,p+2$ prime and 
$p\equiv  1\pmod q$. 
By the usual argument, the number 
of twin primes $p,p+2$ with $p\leq x$ and
$p\equiv 1\pmod q$ equals the number of $t\leq x/q$ such that 
$qt+1$ and $qt+3$ are prime. The number of them is, by the Brun sieve, 
\begin{equation*}
\pi_2(x;q, 1) \ll \frac{x}{(q-1)(\log x)^2}.
\end{equation*}
The Prime Number Theorem and Abel summation reveal that
\begin{equation*}
\sum_{\substack{p\leq x\\ \text{$p,p+2$ prime}}} F_{\epsilon}(p)\ll \frac{x}{(\log x)^2}\sum_{q>  \frac{1}{\epsilon}}\frac{1}{(q-1)^2}\ll \frac{\epsilon \pi_2(x)}{\log(\frac{1}{\epsilon})}.
\end{equation*}
If we let
\begin{equation*}
{\mathcal A}_{\epsilon}=\{p: \text{$p,p\!+\!2$ prime and $F_{\epsilon}(p)>\epsilon$}\},
\end{equation*}
then
\begin{equation*}
 \#{\mathcal A}_{\epsilon}(x) \epsilon 
\,\,\, \leq\!\!\!\! \sum_{\substack{p\leq x\\ \text{$p,p\!+\!2$ prime}}} F_{\epsilon}(p)
 \ll \epsilon \bigg(\log\Big(\frac{1}{\epsilon}\Big) \bigg)^{-1}\pi_2(x),
\end{equation*}
which gives $\#{\mathcal A}_{\epsilon}(x)=O((\log(\frac{1}{\epsilon}))^{-1} \pi_2(x))$. 
\end{proof}

To justify our conjecture, we look at the $\frac{1}{\epsilon}$-part of $p^2-1$. We first let $\epsilon\leq 0.5$. We note that $2| (p-1)$, $2| (p+1)$ and $3| (p+1)$ 
for all twin primes $p\geq 5$. For two coprime square-free numbers $a,b$  with $5\leq p(ab)\leq P(ab)\leq \frac{1}{\epsilon}$, we say that the twin prime $p$ is of {\it $\frac{1}{\epsilon}$-type $(a,b)$} if 
\begin{equation*}
p-1  =  2^{\alpha} \prod_{q\mid a} q^{\alpha_q} \prod_{q>\frac{1}{\epsilon}} q^{\gamma_q}
\qquad \text{and} \qquad
p+1 =  2^{\beta} 3^{\gamma} \prod_{q\mid b} q^{\beta_q} \prod_{q>\frac{1}{\epsilon}} q^{\delta_q}
\end{equation*}
for some positive $\alpha,\beta,\gamma,\alpha_q$ and $\beta_q$ for $q\mid ab$ and nonnegative $\gamma_q,\delta_q$ for $q\geq \frac{1}{\epsilon}$. That is, 
the prime factors of $p-1$ that are $\leq \frac{1}{\epsilon}$ are exactly the ones dividing $2a$ and the prime factors of $p+1$ that are $\leq \frac{1}{\epsilon}$ are exactly the ones dividing $6b$. 

Given $\epsilon$ and $(a,b)$, let
\begin{equation*}
c_{a,b}=\prod_{\substack{5\leq q\leq \frac{1}{\epsilon}\\ q\nmid ab}} q.
\end{equation*}
Note that
\begin{equation*}
\frac{\phi(p-1)}{p-1}=\frac{1}{2} \frac{\phi(a)}{a} \prod_{\substack{q|( p-1)\\ q>\frac{1}{\epsilon}}}\left(1-\frac{1}{q}\right)
\qquad \text{and} \qquad 
\frac{\phi(p+1)}{p+1}=\frac{1}{3} \frac{\phi(b)}{b}
\prod_{\substack{q|(p+1)\\ q>\frac{1}{\epsilon}}} \left(1-\frac{1}{q}\right).
\end{equation*}
Since
\begin{equation*}
e^{-2y}<1-y<e^{-y} \qquad \text{for $y < \tfrac{1}{2}$},
\end{equation*}
it follows that 
\begin{equation*}
1-4\epsilon<e^{-2\epsilon}<e^{-F_{\epsilon}(p)}
=\prod_{\substack{q|( p-1)\\ q>\frac{1}{\epsilon}}} \left(1-\frac{1}{q}\right)
\end{equation*}
hold for all twin primes $p\leq x$ except the ones in ${\mathcal A}_{\epsilon}(x)$, a set of cardinality $O((\log(\frac{1}{\epsilon})^{-1}\pi_2(x))$.  Consequently,
\begin{equation*}
(1-4\epsilon) \frac{\phi(a)}{2a}\leq \frac{\phi(p-1)}{p-1}
\end{equation*}
holds for all but $O((\log(\frac{1}{\epsilon}))^{-1}\pi_2(x))$ twin primes $p\leq x$.  
Thus, the inequality 
\begin{equation*}
\frac{\phi(p-1)}{p-1}\leq \frac{\phi(p+1)}{p+1}
\end{equation*}
implies that
\begin{equation*}
\frac{\phi(a)}{2a}\leq (1-4\epsilon)^{-1}\frac{\phi(b)}{3b}.
\end{equation*}

Let us consider twin primes for which
\begin{equation}\label{eq:2}
\frac{\phi(b)}{3b}<\frac{\phi(a)}{2a}<(1-4\epsilon)^{-1} \frac{\phi(b)}{3b}
\end{equation}
occurs.  Since 
\begin{equation*}
\frac{\phi(a)}{2a}=\frac{\phi(p-1)}{p-1}(1+O(\epsilon))\qquad {\text{\rm and}}\qquad \frac{\phi(b)}{3b}=\frac{\phi(p+1)}{p+1}(1+O(\epsilon))
\end{equation*}
for all $p\leq x$ with $O((\log(\frac{1}{\epsilon}))^{-1}\pi_2(x))$ exceptions, it follows that twin primes $p\leq x$ for which \eqref{eq:2} holds have the additional property that
\begin{equation}
\label{eq:3}
\left|\frac{\phi(p-1)}{p-1}-\frac{\phi(p+1)}{p+1}\right|=O(\epsilon).
\end{equation}
Let ${\mathcal B}_{\epsilon}$ be the set of twin primes for which \eqref{eq:3} holds. We make the following additional assumption.

\medskip\noindent
{\bf Additional assumption}: {\it The number of twin primes $p\leq x$ for which \eqref{eq:3} holds is $O(h(\epsilon) \pi_2(x))$ 
for some function $h(y)$ with $h(y)\to 0$ as $y\to 0$}.

\medskip

The assumption \eqref{eq:3} has been shown to hold when $p$ is only a prime \cite{GaLu}. 
That is, the number of primes $p\leq x$ such that \eqref{eq:3} holds 
is at most $O(h(\epsilon)\pi(x))$, where $h(\epsilon)$ tends to zero when $\epsilon\to 0$.  
In fact, this was a crucial step in showing that $\phi(p-1)-\phi(p+1)$ has no bias if only $p$ is
assumed to be prime.

Proving this for primes uses the Turan--Kubilius theorem about the number of prime factors $q\leq y$ of $p\pm 1$ when $p$ is prime as the parameter $y$ tends to infinity and also Sperner's theorem from combinatorics. With some nontrivial effort, which involves proving first a Turan--Kubilius estimate for the number of distinct primes $q\leq 1/\epsilon$ of $p-1$ and $p+1$ when $p$ ranges over twin primes up to $x$, the same program can be applied to prove that the additional assumption holds under the Bateman--Horn conjectures. We do not give further details here.

Assume that the additional assumption holds. Then the set of twin primes $p\leq x$ such that 
\begin{equation*}
\frac{\phi(p-1)}{p-1}<\frac{\phi(p+1)}{p+1}
\end{equation*}
is within a set of cardinality $O(h(\epsilon)\pi_2(x))$ from the set of primes for which 
\begin{equation}\label{eq:4}
\frac{\phi(a)}{2a}<\frac{\phi(b)}{3b}.
\end{equation}
With this assumption, we proceed as in \cite[Sect.~2.11]{GaLu}. 
Fix $\frac{1}{\epsilon}$, $a$, $b$, and $c=c_{a,b}$. We also fix a residue class for $p$ modulo $c$ which is not $\{0,\pm 1,-2\}$. In this case we need to count natural numbers of the form
\begin{equation*}
abct+\kappa,
\end{equation*}
in which $\kappa$ is fixed such that 
\begin{itemize}
\item $abct+\kappa\leq x$,
\item $abct+\kappa$ and $abct+\kappa+2$ are prime,
\item $abct+\kappa-1$ are divisible by all primes in $a$ and coprime to $cb$,
\item $abct+\kappa+1$ is divisible by all primes in $b$ (and coprime to $ca$).
\end{itemize}
Observe that $\kappa$ is uniquely determined 
modulo $abc$ once it is determined modulo $c$. By the Bateman--Horn conjecture, this number is 
\begin{equation*}
(1+o(1))\pi_2(x)\prod_{p\mid abc} \frac{1}{(p-2)}.
\end{equation*}
We next sum this over all $q-4$ progressions modulo $q$ for which $abct+\kappa$ is not congruent 
modulo $q$ to some member of $\{0,\pm 1,-2\}$ and for all $q\mid c$ 
getting an amount of 
\begin{equation*}
(1+o(1)) \pi_2(x) \prod_{p\mid ab} \left(\frac{1}{p-2}\right) \prod_{p\mid c} \left(\frac{q-4}{q-2}\right)=(1+o(1))
\!\!\!\!\prod_{5\leq q\leq \frac{1}{\epsilon}} \left(\frac{q-4}{q-2}\right)
\prod_{q\mid ab} \left(\frac{1}{q-4}\right).
\end{equation*}
We now sum up over all pairs $a,b$ with 
\begin{equation*}
\frac{\phi(a)}{2a}<\frac{\phi(b)}{3b},
\end{equation*}
which yields a proportion of 
\begin{equation*}
(1+o(1))\prod_{5\leq q\leq \frac{1}{\epsilon}} \left(\frac{q-4}{q-2}\right)
\!\!\!\!\!\!\!\!\!\!\!
\sum_{\substack{a,b\\ 5\leq p(ab)\leq P(ab)\leq \frac{1}{\epsilon}\\ \frac{\phi(a)}{2a}<\frac{\phi(b)}{3b}}} 
\!\!\!\!\!\!\!\!\mu^2(ab) \prod_{p\mid ab} \left(\frac{1}{p-4}\right)
\end{equation*}
of $\pi_2(x)$ with a number of exceptions $p\leq x$ of counting function $O(h(\epsilon) \pi_2(x))$. This supports Conjecture \ref{Conjecture:Formula}.

\section{Comments}

We did not need the full strength of the Bateman--Horn conjecture, just the case $r=2$ and $D=1$ for certain specific pairs of linear polynomials 
$f_1(t)$ and $f_2(t)$.  Under this conjecture, we have seen that $\phi(p-1) \leq \phi(p+1)$ for a substantial majority of twin prime pairs
$p,p+2$.  

There are a few twin primes $p,p+2$ for which
\begin{equation}\label{eq:Equality}
\phi(p-1) = \phi(p+1).
\end{equation}
For only such $p \leq 100{,}000{,}000$ are
\begin{align*}
&5,\quad
11,\quad
71,\quad
2591, \quad
208{,}391,\quad
16{,}692{,}551,\quad
48{,}502{,}931, \quad
92{,}012{,}201 , \\
&249{,}206{,}231,\quad
419{,}445{,}251,\quad
496{,}978{,}301.
\end{align*}
The following result highlights the rarity of these twin primes.

\begin{theorem}
The number of primes $p\leq x$ with $p+2$ prime and  $\phi(p-1)=\phi(p+1)$ is 
$O(x/\exp((\log x)^{1/3})$.  
\end{theorem}

\begin{proof}
Suppose that $j$ and $j+k$ have the same prime factors,
let $g = (j,j+k)$, and suppose that
\begin{equation}\label{eq:jgr}
\frac{j}{g}r + 1 \qquad \text{and} \qquad \frac{j+k}{g}r + 1
\end{equation}
are primes that do not divide $j$.  Then
\begin{equation}\label{eq:Special}
n = j \left( \frac{j+k}{g}r + 1 \right)
\end{equation}
satisfies $\phi(n) = \phi(n+k)$ \cite[Thm.~1]{Graham}.  For $k$ fixed, the number of solutions 
$n \leq x$ to $\phi(n) = \phi(n+k)$ which are not of the form \eqref{eq:Special}
is less than $x/\exp( (\log x)^{1/3})$ for sufficiently large $x$ \cite[Thm.~2]{Graham}.

We are interested in the case $k = 2$ and $n = p-1$, in which $p,p+2$ are prime.
If $j$ and $j+2$ have the same prime factors, then
they are both powers of $2$.  Thus, $j=2$ and $j+k = 4$, so $g= 2$.  From \eqref{eq:jgr}
we see that $r$ is such that
\begin{equation*}
r+1 \qquad \text{and} \qquad 2r+1
\end{equation*}
are prime.  Then
$n = 2(2r+1) = p-1$, from which it follows that $p = 4r+3$ and $p+2 = 4r+5$ are prime.
Consequently,
\begin{equation*}
r+1, \qquad
2r+1,\qquad
4r+3,\quad\text{and}\quad
4r+5,
\end{equation*}
are prime.  However, this occurs only for $r=2$ since otherwise one of the preceding is a multiple of $3$
that is larger than $3$.
\end{proof}

In particular, the number of primes $p \leq x$ for which $p+2$ is prime and 
$\phi(p-1)=\phi(p+1)$ is $o(x/(\log x)^2)$.  Assuming the first Hardy--Littlewood conjecture,
 it follows that the set of such primes has density zero in the twin primes.

\medskip
\noindent\textbf{Acknowledgments.}
We thank Tom\'as Silva for independently computing the ratio $\pi_e(x)/\pi_2(x)$ for large $x$.
We also thank the anonymous referee for suggesting the approach of Section \ref{Section:Conjecture}.

\bibliographystyle{plain}

\bibliography{PRBTP}

\end{document}